\documentclass[11pt]{amsart}

\pdfoutput=1

\usepackage{cancel}
\usepackage{esint,amssymb} 
\usepackage{graphicx}
\usepackage{MnSymbol}
\usepackage{mathtools} 
\usepackage[colorlinks=true, pdfstartview=FitV, linkcolor=blue, citecolor=blue, urlcolor=blue,pagebackref=false]{hyperref}
\usepackage{microtype}

\usepackage{bm}
\usepackage{dsfont}
\usepackage{mathrsfs}
\usepackage{xcolor}

\newtheorem{theorem}{Theorem}

\theoremstyle{remark}

\theoremstyle{definition}

\let\oldH=\H

\newcommand{\R}{\mathbb{R}}

\newcommand{\E}{\mathbb{E}}

\newcommand{\ep}{\varepsilon}

\renewcommand{\le}{\leqslant}
\renewcommand{\ge}{\geqslant}

\newcommand{\Ll}{\left}
\newcommand{\Rr}{\right}

\newcommand{\1}{\mathbf{1}}

\newcommand{\msc}{\mathscr}

\newcommand{\be}{\beta}

\begin{document}

\author{Jean-Christophe Mourrat}
\address[Jean-Christophe Mourrat]{Department of Mathematics, ENS Lyon and CNRS, Lyon, France}

\subjclass[2010]{}
\date{\today}

\title[Color symmetry breaking in the Potts spin glass]{Color symmetry breaking \\ in the Potts spin glass}

\begin{abstract}
The Potts spin glass is an analogue of the Sherrington-Kirkpatrick model in which each spin can take one of $\kappa$ possible values, which we interpret as colors. It was suggested in \cite{bates2024parisi} that the order parameter for this model is always invariant with respect to permutations of the colors. We show here that this is false whenever $\kappa \ge 58$. 
\end{abstract}

\maketitle

%
%
%
%
%
%

Let $\kappa \ge 2$ and $N\ge 1$ be integers, and let $(g_{ij})_{i,j \ge 1}$ be independent centered Gaussians of variance $1$. The energy function of the Potts spin glass is defined, for every $\sigma \in \{1,\ldots, \kappa\}^N$, by
\begin{equation}
\label{e.def.HN}
H_N(\sigma) := \frac 1 {\sqrt{N}} \sum_{i,j = 1}^N g_{ij} \1_{\{\sigma_i = \sigma_j\}}.
\end{equation}
The associated free energy at inverse temperature $\beta \ge 0$ is 
\begin{equation}  
\label{e.def.FN}
F_N(\beta) := \frac 1 N \E \log \sum_{\sigma \in \{1,\ldots, \kappa\}^N}\exp (\beta H_N(\sigma)).
\end{equation}
For $\kappa  = 2$, the Potts spin glass essentially coincides with the Sherrington-Kirkpatrick model \cite{sherrington1975solvable}. Indeed, for $\kappa = 2$ we may as well consider that $\sigma$ ranges in $\{-1,1\}^N$, and for every $\sigma \in \{-1,1\}^N$ and $i,j \in \{1,\ldots, N\}$, we can write
\begin{equation*}  
\1_{\{\sigma_i = \sigma_j\}} = \frac 1 2 (\sigma_i \sigma_j + 1),
\end{equation*}
so that 
\begin{multline}  
\label{e.identity.free.energies}
\frac 1 N\E \log \sum_{\sigma \in \{1,2\}^N} \exp(\beta H_N(\sigma)) 
\\
= \frac 1 N\E \log \sum_{\sigma \in \{-1,1\}^N} \exp \Ll(\frac{\beta}{2\sqrt{N}} \sum_{i,j = 1}^N g_{ij} \sigma_i \sigma_j\Rr).
\end{multline}
The right-hand side of the display above is the free energy of the Sherrington-Kirkpatrick model at inverse temperature $\beta/2$.

The asymptotic behavior of the free energy of the Potts spin glass has been obtained in \cite{pan.potts}. The analysis proceeds by first identifying the asymptotics of the free energy restricted to configurations with a prescribed proportion of each color. To be precise, let 
\begin{equation*}  
\msc D := \Ll\{(d_1,\ldots, d_\kappa) \in [0,1]^\kappa \ \mid \ \sum_{k = 1}^\kappa d_k = 1\Rr\},
\end{equation*}
and for each $d \in \msc D$ and $\ep > 0$, let
\begin{multline*}  
\Sigma_N(d,\ep) 
\\
:= \Ll\{ \sigma \in \{1,\ldots, \kappa\}^N \ \mid \  \text{for every } k \in \{1,\ldots, \kappa\}, \  \Ll| \frac 1 N\sum_{i =1}^N \1_{\{\sigma_i = k\}} - d_k  \Rr| \le \ep \Rr\},
\end{multline*}
\begin{equation*}  
F_N(\beta,d,\ep) := \frac 1 N \E \log \sum_{\sigma \in \Sigma_N(d,\ep)}\exp (\beta H_N(\sigma)) .
\end{equation*}
We also define, for every $d \in \msc D$,
\begin{multline}  
\label{e.def.pid}
\Pi_d := \Big\{ \pi : [0,1] \to S^\kappa_+ \ \mid \ \pi \text{ is left-continuous, non-decreasing, } 
\\
\pi(0) = 0, \text{ and } \pi(1) = \mathrm{diag}(d_1,\ldots, d_\kappa)  \Big\} ,
\end{multline}
where $S^\kappa_+$ denotes the set of $\kappa$-by-$\kappa$ positive semidefinite matrices. In \eqref{e.def.pid}, we say that the path $\pi$ is non-decreasing to mean that for every $u\le v \in [0,1]$, we have $\pi(v) - \pi(u) \in S^\kappa_+$.  
\begin{theorem}[\cite{pan.potts}]
\label{t.limit.free}
There exists an explicit functional $\msc P_\beta : \bigcup_{d \in \msc D} \Pi_d \to \R$ such that for every $d \in \msc D$,
\begin{equation*}  
\lim_{\ep \to 0} \limsup_{N \to \infty} F_N(\beta,d,\ep) = \lim_{\ep \to 0} \liminf_{N \to \infty} F_N(\beta,d,\ep) = \inf_{\pi \in \Pi_d} \msc P_\beta(\pi).
\end{equation*}
As a consequence,
\begin{equation}  
\label{e.lim.FN}
\lim_{N \to +\infty} F_N(\beta) = \sup_{d \in \msc D} \inf_{\pi \in \Pi_d} \msc P_\beta(\pi). 
\end{equation}
\end{theorem}
The reference \cite{pan.potts} provides us with an explicit description of the functional~$\msc P_\beta$ as an infimum over an additional parameter denoted by $\lambda$ there. For the purposes of this note, we will only need an upper bound on $\msc P_\beta$ which is obtained by selecting $\lambda = 0$, and we will only need this upper bound on very simple (replica-symmetric) paths.  

The definition of the Potts spin glass is clearly invariant under permutations of the $\kappa$ different values that a spin can take. We interpret these different values as colors. We say that the Potts spin glass (at inverse temperature $\beta$) \emph{preserves the color symmetry} if the supremum over $d \in \msc D$ in \eqref{e.lim.FN} is achieved at $d = (\frac 1 \kappa, \cdots, \frac 1 \kappa)$, and if moreover, the infimum over $\pi \in \Pi_{(\frac 1 \kappa, \cdots, \frac 1 \kappa)}$ in \eqref{e.lim.FN} is achieved at a path~$\pi$ that is color-symmetric (in other words, for each $u \in [0,1]$, the diagonal entries of $\pi(u)$ are all the same, and the non-diagonal entries of $\pi(u)$ are all the same). Otherwise, we say that the Potts spin glass \emph{breaks the color symmetry}. The main result of \cite{bates2024parisi} is that the infimum over $\pi \in \Pi_{(\frac 1 \kappa, \cdots, \frac 1 \kappa)}$ in \eqref{e.lim.FN} is always achieved at a path~$\pi$ that is color-symmetric. In particular, we could equivalently say that the Potts spin glass (at inverse temperature $\beta$) preserves the color symmetry if the supremum over $d \in \msc D$ in \eqref{e.lim.FN} is achieved at $d = (\frac 1 \kappa, \cdots, \frac 1 \kappa)$. It was suggested in \cite{bates2024parisi} that the Potts spin glass always preserves the color symmetry. We postpone a more precise discussion of the literature and first show that this conjecture is false when~$\kappa$ is sufficiently large. 

\begin{theorem}[Color symmetry breaking]
\label{t.main}
For every $N \ge 1$ and $\beta \ge 0$, we have
\begin{equation}  
\label{e.main1}
F_N(\beta) \ge \Ll( \frac{N-1}{N} \Rr)^{3/2} \frac{2\beta}{3 \sqrt{\pi}},
\end{equation}
and
\begin{equation}  
\label{e.main2}
\inf_{\pi \in \Pi_{(\frac 1 \kappa, \ldots, \frac 1 \kappa)}} \msc P_\beta(\pi) \le \log \kappa +  \frac{\beta^2}{2\kappa} .
\end{equation}
In particular, the claim that $F_N(\beta)$ converges to the left-hand side of \eqref{e.main2} is false as soon as 
\begin{equation}  
\label{e.crit}
\frac{\kappa}{\log \kappa} > \frac{9\pi}{2} \quad \text{ and } \quad \Ll|\frac{3\sqrt{\pi}}{2\kappa} \beta -  1\Rr|^2 < 1- \frac{9\pi \log \kappa}{2\kappa}. 
\end{equation}
\end{theorem}
\begin{proof}
We start with the proof of \eqref{e.main1}. By restricting the summation on $\sigma \in \{1,\ldots, \kappa\}^N$ to a summation over $\sigma \in \{1,2\}^N$ in the definition of $F_N(\beta)$ in~\eqref{e.def.FN}, we see that the term on the left side of \eqref{e.identity.free.energies} is a lower bound for $F_N(\beta)$. A simple lower bound for the term on the right side of \eqref{e.identity.free.energies} can be found in \cite[Exercises~6.1 and 6.3 and solutions]{HJbook}, and this yields \eqref{e.main1}. 

As announced, for the proof of \eqref{e.main2} we only need to consider a very special path, and we fix the additional parameter $\lambda$ appearing in~\cite{pan.potts} to be zero. We choose the path $\pi$ to be constant equal to $0$
 over $[0,1)$. In the notation of \cite{pan.potts}, this corresponds to the case of $r = 1$, $x_0 = 1$, $\gamma_0 = 0$, $\gamma_1 = \mathrm{diag}(\frac 1 \kappa, \ldots, \frac 1 \kappa)$.
Letting $(z_1, \ldots, z_\kappa)$ be independent centered Gaussians of variance $1$, we find that
\begin{equation*}  
\msc P_\beta(\pi) \le  \log \E\sum_{k = 1}^\kappa\exp \Ll( \beta\sqrt\frac{2}{\kappa}  z_k \Rr) - \frac{\beta^2}{2\kappa} = \log \kappa + \frac{\be^2}{2\kappa},
\end{equation*}
which is \eqref{e.main2}. The last part of the theorem follows by identifying the region in which the right-hand side of \eqref{e.main1} exceeds the right-hand side of~\eqref{e.main2}. 
\end{proof}
The condition in \eqref{e.crit} is non-empty as soon as $\kappa \ge 58$. I made no attempt to obtain sharp bounds. In particular, the argument for~\eqref{e.main1} yields that the large-$N$ limit of~$F_N(\beta)$ is bounded from below by $\beta/2$ times the maximum of the Sherrington-Kirkpatrick energy function, which is expected to be about $\sqrt{2} \times 0.763\ldots$ \cite{crisanti2002analysis, el2020algorithmic, schmidt2008replica} (the extra $\sqrt{2}$ accounts for a difference in the choice of normalization). Using this bound instead, we find that color symmetry breaking occurs as soon as the number of colors $\kappa$ is at least $21$. 

By analogy with the non-disordered version of the Potts model, one may expect that, at least for sufficiently large values of $\kappa$, the range of $\beta$'s at which color symmetry is broken is unbounded. The bound \eqref{e.main2} is however too crude to allow us to obtain this. By reasoning as in \cite[Exercise~6.3 and solution]{HJbook}, one can see that for every $N \ge 1$, $d \in \msc D$ and $\ep > 0$, 
\begin{equation*}  
F_N(\beta,d,\ep) = \frac{\beta}{N} \E \max_{\sigma \in \Sigma_N(d,\ep)} H_N(\sigma) + O\Ll( 1 \Rr) \qquad (\beta \to +\infty),
\end{equation*}
so the bound in \eqref{e.main2} does not even capture the correct asymptotic behavior of this quantity as $\beta$ tends to infinity, as it incorrectly scales like $\beta^2$ instead of scaling like $\beta$. There is clearly a lot of room to improve upon \eqref{e.main2}. 

We now briefly review previous works on the topic. 
Most works in the physics literature also allow for the couplings $(g_{ij})$ to have a bias. In order to discuss this while keeping consistent notation, we thus define a more general version of the free energy by setting, for every $\beta \ge 0$ and $\gamma \in \R$, 
\begin{equation}  
\label{e.def.FN2}
F_N(\beta, \gamma) := \frac 1 N \E \log \sum_{\sigma \in \{1,\ldots, \kappa\}^N}\exp \Ll(\beta H_N(\sigma) + \frac{\beta \gamma}{N} \sum_{i,j = 1}^N \1_{\{\sigma_i = \sigma_j\}}\Rr).
\end{equation}
In the papers \cite{elderfield1983curious, elderfield1983novel}, the authors restrict their analysis to the case $\gamma \le \gamma_F(\kappa)$, where $\gamma_F(\kappa)$ is ``the critical mean exchange for the
highest-temperature transition to be to a ferromagnetic state''. 
In \cite{elderfield1983spin}, they announce a full resolution of the phase diagram in $(\beta, \gamma)$ (also allowing for an external field); I cannot extract from there a precise condition on $(\beta,\gamma)$ that would guarantee color symmetry, but they say that this requires to take $\gamma < 0$ for $\kappa > 4$, citing \cite{lage1983mixed}. A concurrent work is \cite{erzan1983infinite}, which focuses on the case $\gamma = 0$; the authors postulate color symmetry there, but they quickly correct this in~\cite{lage1983mixed}  and propose a more sophisticated solution that is not color-symmetric for $\kappa > 4$. 
In \cite{gross1985mean}, we  read that ``an appropriate nonzero value of $\gamma$ must be chosen \cite{elderfield1983curious, elderfield1983spin, elderfield1983novel}'' (notation and pointers adapted). In~\cite{desantis1995static}, the authors state that ``ferromagnetic order is always preferred for $\kappa> 2$ for sufficiently low temperature'', they cite \cite{elderfield1983spin}, and they give a formula for the transition temperature which they denote by $T_F$. They specify that the transition is to a ``colinear ferromagnet'', which in the language of \cite{elderfield1983spin} is a phase in which only one color displays a non-zero overlap (in particular, this phase is not color-symmetric). They then say that ``In the special case $\gamma = 0$ the ferromagnetic transition appears below $T = 1$ for $\kappa < 4$ and above that temperature for $\kappa > 4$. Our main interest in this
paper is the study of the spin-glass transition. In order not to observe the ferromagnetic transition it will be necessary to add an antiferromagnetic coupling in the case $\kappa > 4$.'' (notation for $\gamma$ and $\kappa$ adapted). For their numerical simulations, they chose $\gamma =  \frac{4-\kappa}{2\sqrt{2}}$ for $\kappa \ge 4$ and found it suitable to their stated needs. Similar statements can also be found in \cite{caltagirone2012dynamical}, although with a different formula for the transition temperature  $T_F$, the existence of which is attributed to \cite{gross1985mean} there. 

The recent paper \cite{bates2024parisi} suggests that color symmetry is preserved for $\gamma = 0$ and arbitrary values of $\kappa$ and $\beta$. As we have seen, this is invalid at least for $\kappa \ge 58$. The authors of \cite{bates2024parisi} attribute this color-symmetry prediction to~\cite{elderfield1983curious}. My own reading of the physics literature is different, as explained in the previous paragaph. 
The work \cite{bates2024parisi} inspired \cite{chen2023parisi}, in which color symmetry is shown for all $\beta$ with the choice of $\gamma = 
-\frac \beta 2$. The results of \cite{chen2023parisi} have been generalized in \cite{issa2024existence} to a broader class of models. One may also consult \cite{chen2024free, pan.vec} for results on the asymptotics of the free energy of more general spin-glass models, and \cite{costeniuc2005complete, ellis1992limit} for a thorough study of the non-disordered mean-field Potts model. We also note that \cite{sen2018optimization} establishes a connection between the maximum $\kappa$-cut of an Erd\oldH os-R\'enyi random graph with average degree $d$ and the maximal energy of the balanced Potts spin glass, in the regime of large $d$.

\small
\bibliographystyle{plain}
\bibliography{potts}

\end{document}